\documentclass[12pt]{article}
\usepackage[T2A]{fontenc}
\usepackage[cp1251]{inputenc}
\usepackage[english,russian]{babel}
\usepackage{graphicx}
\usepackage{amsmath}
\usepackage{color}
\usepackage{amssymb}
\usepackage{chicago}
\usepackage{graphicx}
\usepackage{epstopdf}
\usepackage{epsfig}
\usepackage{multirow}

\newtheorem{theorema}{Теорема}

\newtheorem{ddef}{Определение}

\newtheorem{example}{Пример}
\newtheorem{lemma}{Лемма}

\newtheorem{Remark}{Замечание}[theorema]

\newtheorem{lemm_conseq}{Следствие}[lemma]

\newenvironment{proof}{\par\noindent{\bf Доказательство.}}{\hfill$\scriptstyle\blacksquare$}
\makeatletter
\renewcommand \thesection {\@arabic\c@section.}
\renewcommand\thesubsection {\thesection\@arabic\c@subsection.}
\renewcommand\thesubsubsection {\thesubsection\@arabic\c@subsubsection.}
\renewcommand\theparagraph {\thesubsubsection\@arabic\c@paragraph.}

\begin{document}
\selectlanguage{russian}

\title{
Кратчайшие и минимальные дизъюнктивные нормальные формы полных функций
}

\author{Ю.~В. Максимов\thanks{Исследование выполнено при финансовой поддержке РФФИ в рамках научного проекта \No\,14-07-31277 мол\_а; а также при частичной поддержке Лаборатории структурных методов анализа данных в предсказательном моделировании ФУПМ МФТИ, грант правительства РФ дог. 11.G34.31.0073.
}\\
ПреМоЛаб МФТИ и ИППИ РАН
}
\date{}

\maketitle

\begin{abstract}
Почти всем булевым функциям от $n$ переменных, число нулей $k$ которых не превышает $\log_2 n - \log_2\log_2 n +1$, может быть сопоставлена некоторая булева функция от $2^{k-1} - 1$ переменой с $k$ нулями (\textit{полная} функция) так, что сложность реализации исходной функции в классе дизъюнктивных нормальных форм (ДНФ) определяется исключительно сложностью реализации полной функции. Это было установлено ранее в работах \cite{zk85dan,zk86}.

В данной работе установлена асимптотически точная граница для минимальное число литералов и конъюнкций, содержащихся в ДНФ полной функции.\vspace{0.5cm} \\
\noindent \textbf{Ключевые слова:}
булевы функции, дизъюнктивная нормальная форма, сложность реализации булевых функций дизъюнктивными нормальными формами.
\end{abstract}

\section{Введение}

К задаче построения простых дизъюнктивных нормальных форм~(ДНФ) булевых функций приводит большое число практических задач из различных областей дискретной оптимизации, таких как построение эффективных $SAT$--солверов, вычисление характеристических функций классов в задачах машинного обучения, синтез управляющих схем и многие другие \cite{ka09,zh78,boros00,boros11,kor09,kor12}.

В работе исследуются булевы функции, заданные в конъюнктивной нормальной форме~(КНФ) произведениями вида
\begin{gather}\label{general_eq}
    f(x_1,...,x_n) = \bigwedge\limits_{i=1}^k \left( x_1^{{\alpha}_i^1} \vee x_2^{{\alpha}_i^2} \vee ...\vee x_n^{{\alpha}_i^n}\right)
\end{gather}


Трудоемкость прямого построения сокращенной ДНФ путем перемножения скобок конъюнктивной нормальной формы (КНФ) и последующего ее упрощения до минимальной ДНФ послужила толчком для разработки эффективных методов аппроксимации оптимальных по сложности ДНФ. Как правило, при исследовании сложности результирующей ДНФ оцениваются две основные меры сложности: \textit{длина}, равная числу входящих в нее конъюнкций, и \textit{ранг}, равный суммарному числу входящих в нее литералов. Дизъюнктивная нормальная форма с наименьшим значением ранга называется \textit{минимальной}, а ДНФ с минимальным числом конъюнкций --- \textit{кратчайшей}. Под \textit{линейной} мерой сложности ДНФ мы будем понимать произвольную выпуклую комбинацию длины и ранга. Впервые указанная мера была определена в работе \cite{veb79}. Рассмотрение линейных мер сложности мотивировано неэквивалентностью понятий минимальной и кратчайшей ДНФ, установленной Ю.~И.~Журавлёвым в работе \cite{zmin}.

С.В. Яблонским, Ю.И. Журавлёвым, А.Г. Дьяконовым и другими исследователями было показано, что при достаточно малом числе скобок в произведении \ref{general_eq} возможны эффективные методы построения ДНФ близких к минимальным и кратчайшим \cite{zk85dan,zk86,dj01,dj02}.

В общем случае, сложность преобразования КНФ в ДНФ достаточно велика, что не позволяет осуществить его в пространстве большой размерности~\cite{Golea97}. Тем не менее в работе \cite{zk85dan} было показано, что если число сомножителей в произведении \ref{general_eq} не превосходит $k = \log_2 n - \log_2\log_2 n + 1$, то построение искомой ДНФ сводится к построению ДНФ некоторой функции от существенно меньшего числа переменных и при этом сложность сведения невелика. Более того, почти все исходные функции с указанным числом нулей и переменных сводимы к одной и той же булевой функции, названной в работе \cite{zk85dan} \textit{полной}.

Основным результатом настоящей работы является подход, позволяющий строить асимптотически минимальные ДНФ по всем линейным мерам сложности для широкого класса булевых функций. В частности, предложенные в работе алгоритмы позволяют строить и асимптотически минимальную и асимптотически кратчайшую ДНФ для полной функции.

\section{Терминология}
В данной работе использован ряд стандартных терминов теории булевых функций, определенных в работах \cite{ja10,dj01}. Кратко напомним основные из них.

Матрицей нулей $M_f$ булевой функции $f(x_1,x_2,...,x_n)$ является матрица, по строкам которой последовательно выписаны все различные нули функции $f$. Обозначим $M_{i_1,...,i_t}^{j_1,...,j_s}$ подматрицу матрицы $M$, находящуюся на пересечении строк $i_1,...,i_t$ со столбцами $j_1,...,j_s$. Матрицу, не содержащую одинаковых строк, назовем \textit{тестом}.

Положим $B_k$ --- множество всех булевых векторов размерности $k$; $B_k^0$ и $B_k^1$ --- множество булевых векторов, первая координата которых равна 0 или 1 соответственно; пусть $[k,n] = {k,k+1,k+2,...,n}$ --- целочисленный отрезок от $k$ до $n$ включительно, обозначим $[n]$ отрезок $[1,n]$,
а $\chi:B_k \rightarrow [2^k-1]$ --- (взаимно-однозначное) отображение заданное правилом
    \[\chi\left((\sigma_1,...,\sigma_k)\right) = \sigma_12^{k-1} +\sigma_22^{k-2}+ ... +\sigma_k2^{0}.\]
\quad    Значение $\chi(x)$ назовем \textit{номером}, а минимум из числа единиц и числа нулей назовем \textit{весом} вектора $x, x\in B_k$.

Литералом будем называть булеву переменную или ее отрицание. Литерал $x_i$ ассоциируем со столбцом $M_{[k]}^i$, а литерал $\bar{x_i}$ с --- покоординатным отрицанием столбца $M_{[k]}^i, 1\le i \le n$. Далее под термином литерал $x_i^{\sigma_i}$ будет часто подразумеваться ассоциированный с ним булев вектор. Дизъюнкции (конъюнкции) литералов сопоставим дизъюнкцию(конъюнкцию) ассоциированных с ними булевых векторов.

Отображение $\chi$ индуцирует линейный порядок на векторах из $B_k$ в соответствии с естественным порядком на прообразах. Обозначим его $\prec_\chi$. Обозначим $\prec$ стандартный частичный порядок на $B_k$.

Скажем, что литерал $x_i^{\sigma_i} \prec x_j^{\sigma_j}$($x_i^{\sigma_i} \prec_\chi x_j^{\sigma_j}$) для функции $f$, если соответствующие им столбцы матрицы (или их отрицания) связаны таким же соотношением относительно частичного порядка $\prec$($\prec_\chi$). Под весом и номером литерала будем понимать вес и номер соответствующих ему векторов матрицы нулей.

\begin{ddef}
Ненулевой вектор $\alpha \in B_k$ назовем \textit{разложимым} по векторам $\alpha_1, ...,\alpha_t$, если выполнены следующие равенства
\[
\alpha = \alpha_1 \vee \alpha_2 \vee ... \vee \alpha_t, \qquad <\!\bar{\alpha},\alpha_1\!> = <\!\bar{\alpha}, \alpha_2\!> = ... = <\!\bar{\alpha},\alpha_t\!> = 0,
\]
где под символом $<\!\alpha, \beta\!>$ понимается скалярное произведение векторов $\alpha$ и $\beta$.
\end{ddef}
\begin{ddef}
Ненулевой вектор $\alpha \in B_k$ назовем \textit{ортогонально разложимым} по векторам $\alpha_1, ...,\alpha_t$, если выполнены следующие равенства
\[\begin{cases}
\alpha = \alpha_1 \oplus \alpha_2 \oplus ... \oplus \alpha_t\\
\alpha = \alpha_1 \vee \alpha_2 \vee ... \vee \alpha_t
\end{cases}\]
\end{ddef}

Обозначим $\mathbb{I}_k$ вектор размерности $k$ состоящий только из единиц. В случае $\alpha = \mathbb{I}_k$ в условиях предыдущего определения назовем вектора $\{\alpha_i\}_{i=1}^t$ \textit{ортогональным разложением единицы}. Заметим, что если вектор $\alpha$ разложим по векторам $\{\alpha_i\}_{i\in I}$, то $\chi(\alpha) > \chi(\alpha_i), i\in I$.

В теории сложности ДНФ традиционно рассматриваются такие характеристики, как число литералов литералов и конъюнкций входящих ДНФ. Число литералов ДНФ $D$ обычно обозначается $rank(D)$, число конъюнкций $D$ обозначается $|D|$ или $term(D)$. В работе \cite{veb79} рассматривались произвольные выпуклые\footnote{В оригинальной работе такие меры сложности назывались линейными мерами.} комбинации числа литералов и числа конъюнкций ДНФ, для которых в работе будет использовано обозначение $conv_{\alpha, \beta}(D) = \alpha\cdot size(D) + \beta\cdot term(D), \alpha > 0, \beta > 0, \alpha + \beta = 1$. В настоящей работе будут рассматриваться только описанные выше меры сложности.

Пусть $\pi_n$ --- класс (инвариантных) преобразований Шеннона-Поварова булевых функций от $n$ переменных, таких, что всякое преобразование $\pi \in \pi_n$ сопоставляет упорядоченному набору переменных $(x_1,x_2,...,x_n)$ набор $(y_1,y_2,...,y_n) = (x_{i_1}^{\sigma_{i_1}},x_{i_2}^{\sigma_{i_2}},...,x_{i_n}^{\sigma_{i_n}})$, где $(i_1,...,i_n)$ --- произвольная перестановка отрезка $[1,n]$, $\sigma_i \in \{0,1\}$ при $i\in[1,n]$. Отметим, что булева функция и ее образ при преобразовании Шеннона-Поварова имеют одинаковую сложность при представлении их дизъюнктивными нормальными формами.

\textit{Цепью} называется монотонная(в смысле отображения $\chi$) упорядоченная последовательность точек из $B_k$ без повторений в которой каждая следующая точка отличается от предыдущей только в одной координате. Для последовательности точек $\alpha_1,\alpha_2,\alpha_3$ входящей в цепь $C$, назовем точку $\alpha$ \textit{дополнительной}, если $\alpha_1, \alpha,\alpha_3$ цепь и, кроме того, $\alpha_2\neq\alpha$. Цепь назовем \textit{симметричной}, если ее первый и последний элементы имеют одинаковый вес.

Обозначим $Z_f$ --- множество нулей функции $f$, а $N_f$ --- множество ее единиц.

Определим вектора $\{e_i^k\}_{i=1}^{k}$ равенствами $e_{i+1}^k = \chi^{-1}(2^i), 0 \le i < k$.

Конъюнкция $K$ называется импликантой функции $f(x_1,x_2,...,x_n)$ если $N_K \subseteq N_f$. Импликанта $K$ называется простой, если из $K$ не может быть вычеркнут ни один литерал, так чтобы полученная конъюнкция была импликантой $f(x_1,x_2,...,x_n)$. Термины \textit{простая импликанта} $f$ и \textit{несократимая конъюнкция} входящая в сокращенную ДНФ функции $f$ употребляются в работе как синонимы.

Обозначим $P^n_k$ класс булевых функций от $n$ переменных, имеющих в точности $k$ различных нулей. Произвольной функции $f \in P_k^n$ сопоставим функцию $g \in P_k^l$ таким образом, что
\begin{enumerate}
\item   в матрице $M_f$ отсутствуют нулевой и единичный столбцы;
\item   одинаковые столбцы в $M_g$ расположены последовательно;
\item   для любой пары столбцов, один из которых является отрицанием другого, в матрице $M_g$ присутствует не более одного.
\end{enumerate}

Заметим, что инвариантными преобразованиями Шеннона-Поварова всякая булева функция, не имеющая столбцов веса 0, может быть записана в таком виде(с сохранением сложности в классе ДНФ).
Булева функция $g$ называется \textit{правильной}, если ее матрица нулей удовлетворяет условиям~$1-3$. Если, кроме того, матрица нулей $g$ не содержит одинаковых столбцов, то $g$ называется \textit{приведенной}.

Приведенную функцию, принадлежащую классу $P_k^{2^{k-1}-1}$ назовем \textit{полной}. Обозначим $\mathbb{F}_k$ класс всех полных функций, имеющих ровно $k$ нулей. Полную булеву функцию класса $\mathbb{F}_k$, столбцы матрицы нулей которой упорядочены по возрастанию в соответствии с отображением $\chi$, назовем канонической (полной) функцией и обозначим ее $F_k$.

Пусть $ind$ --- функция, сопоставляющая всякому бинарному вектору число его ненулевых компонент. Далее нам понадобится еще одна полная функция $G_k$, такая что столбец матрицы нулей $G_k$ совпадает с соответствующим столбцом матрицы нулей $F_k$, если число единиц в последнем не больше чем $k/2$, и дополнителен к столбцу матрицы нулей $F_k$ в противном случае.

Известно \cite{zk85dan,zk86}, что всякая правильная булева функция $\phi$ представима в виде
\begin{multline}\label{formula_reduction}
    D_{\phi}(x_1, x_2, ..., x_n) = D_F(x_{i_1}, x_{i_2}, ...x_{i_t}) \vee D_2(x_{i_1}, x_{i_1+1}, ..., x_{i_2-1}) \vee ... \vee D_2(x_{i_t}, x_{i_t+1}, ..., x_{i_{t+1}-1})
\end{multline}
где $i_1 = 1$, а $i_{t+1} = n+1$; столбцы с номерами $i_s, i_{s+1}, ..., i_{s+1}-1$ равны при $1\le s \le t$; набор столбцов с номерами $i_1,i_2, ..., i_t$ состоит в точности из $t$ различных столбцов. Здесь $D_F $ в формуле (\ref{formula_reduction}) --- ДНФ соответствующей $\phi$ приведенный функции $F$, а $D_2(x_{i_s}, x_{i_{s+1}}, ..., x_{i_{s+1} - 1}) = x_{i_s} \bar{x}_{i_s + 1} \vee x_{i_s+1} \bar{x}_{i_s + 2} \vee ... \vee x_{i_{s+1}-1} \bar{x}_{i_s}$ при $1\le s \le t$. Отметим, что указанным способом можно сводить построение ДНФ произвольной булевой функции к построению ДНФ соответствующей приведенной функции при этом сложность сведения линейна.

    Пусть далее $\mu$ --- равномерная вероятностная мера на множестве булевых функций из $P_k^n$. В работах \cite{zk85dan,zk86} была установлена следующая  теорема
\begin{theorema}
Почти всем правильным булевым функциям от $n$ переменных, которые имеют не более $k < \log_2 n - \log_2 \log_2 n + 1$ нулей, при $n\rightarrow\infty$ по формуле \ref{formula_reduction} ставится в соответствие полная булева функция.
\end{theorema}

    В настоящей работе предполагается, что всякая рассматриваемая функция задана явно своей матрицей нулей.


\begin{example}
Рассмотрим правильную булеву функцию $\phi(x_1, x_2, ..., x_{12})$ заданную матрицей
\[M_{\phi} = \left(
              \begin{array}{cccccccccccc}
                1 & 1 & 0 & 0 & 0 & 1 & 1 & 0 & 1 & 1 & 0 & 0\\
                1 & 1 & 1 & 1 & 1 & 1 & 1 & 0 & 0 & 0 & 0 & 1\\
                1 & 1 & 1 & 1 & 1 & 0 & 0 & 1 & 0 & 1 & 1 & 0\\
                0 & 0 & 0 & 0 & 0 & 0 & 0 & 0 & 0 & 1 & 0 & 1\\
              \end{array}
            \right)\]

Указанной функции могут быть сопоставлены $2^7\cdot7!$ полных функций с различными матрицами нулей, в том числе $F_4$ и $G_4$ с матрицами нулей
\[M_{F_4} = \left(
              \begin{array}{ccccccc}
                0 & 0 & 0 & 0 & 0 & 0 & 0 \\
                0 & 0 & 0 & 1 & 1 & 1 & 1 \\
                0 & 1 & 1 & 0 & 0 & 1 & 1 \\
                1 & 0 & 1 & 0 & 1 & 0 & 1 \\
              \end{array}
            \right) \qquad M_{G_4} = \left(
              \begin{array}{ccccccc}
                0 & 0 & 0 & 0 & 0 & 0 & 1 \\
                0 & 0 & 0 & 1 & 1 & 1 & 0 \\
                0 & 1 & 1 & 0 & 0 & 1 & 0 \\
                1 & 0 & 1 & 0 & 1 & 0 & 0 \\
              \end{array}
            \right)
\]
а ДНФ функции $\phi$ представима в виде
\begin{multline}
D_{\phi(x_1, x_2, ..., x_{12})} = (x_1\bar{x}_2 \vee x_2\bar{x}_1) \vee (x_3\bar{x}_4 \vee x_4\bar{x}_5 \vee x_5\bar{x}_3) \vee (x_6\bar{x}_7 \vee x_7\bar{x}_6) \vee D_{F_4(\bar{x}_1, x_8, \bar{x}_6, \bar{x}_{10}, x_{12}, x_3, x_9)}.
\end{multline}
где $D_{F_4}$ --- ДНФ соответствующей полной функции.
\end{example}

\section{Известные подходы и результаты}
В работах Ю.И. Журавлёва и А.Ю. Когана были предложены исторические первые редукционный и аппроксимационной алгоритмы построения ДНФ булевых функций с малым числом нулей \cite{zk85dan,zk86}. Наилучшими теоретическими оценками длины и ранга ДНФ из предложенных обладает редукционный алгоритм. Длина построенной редукционным алгоритмом ДНФ для полной функции асимптотически равна $2n(1+ o(1))$, а ее ранг равен $4n(1+ o(1))$. Вычислительно более эффективная модификация этого метода предложена в работах \cite{me12a,me13}.

Существенным недостатком предложенного алгоритма является его относительная трудоемкость и высокие требования к памяти. В частности, при построении ДНФ редукционным алгоритмом после каждой произведенной итерации необходимо было сохранять все множество промежуточных результатов. Таким образом, при построении ДНФ полной функции от $n$ переменных требовалось как минимум $\Omega(n\log_2^2n)$ бит памяти для хранения элементов матриц получаемых в процессе работы алгоритма.

Позднее А.Г. Дьяконовым в работах \cite{dj01,dj02} был предложен существенно более эффективный в смысле вычислительной сложности алгоритм, позволяющий строить ДНФ полной функции длины асимптотически равной $n(1+o(1))$, однако ранг построенной ДНФ был равен $\Theta(n\sqrt{\log_2 n})$. Таким образом, несмотря на относительно простую процедуру обучения, время работы методов, основанных на ДНФ, построенных этим алгоритмом может быть достаточно было велико.

Также А.Г. Дьяконовым в работе \cite{dj02} был предложен тестовый подход к построению ДНФ. Указанный подход оказывается успешным для построения ДНФ булевых функций с малым числом нулей достаточно общего вида, однако теоретические верхние оценки длины и ранга построенных им ДНФ, в частности ДНФ полной функции, оказываются зачастую сильно завышенными. Интересное обобщение предложенного А.Г. Дьяконовым метода дано в работе~\cite{mu06}.

В настоящей работе предложен алгоритм построения ДНФ полной функции, доставляющий (асимптотически) минимальные известные оценки и числа литералов, и числа конъюнкций ДНФ полной функции. Число литералов в построенной им ДНФ равно $3n(1+o(1))$, а число конъюнкций равно $n(1+o(1))$. Показано, что все указанные оценки асимптотически точны при $n\rightarrow\infty$. Кроме того, предлагаемый алгоритм построения ДНФ полной функции, требует сравнительно небольшого объем памяти, что позволяет использовать его при решении прикладных задач. Предварительная версия данного метода анонсирована автором в кратком сообщении~\cite{me12b}.

\section{Нижняя оценка ранга ДНФ полной функции}
\textit{Соседними} назовем булевы вектора одинаковой размерности, находящиеся на расстоянии 1 в метрике Хемминга. Рассмотрим произвольную приведенную булеву функцию $f(x_1,...,x_n)$ не имеющую соседних нулевых точек; обозначим $M_f$ ее матрицу нулей и зафиксируем ее произвольную ДНФ $D$.

Будем считать, что всякая рассматриваемая далее конъюнкция несократимая, то есть является простой импликантой $f$. Аналогично будем полагать, что всякая рассматриваемая в этом параграфе ДНФ состоит исключительно из простых импликант. Кроме того, будем полагать, что каждая из рассматриваемых далее матриц является тестом.

В процессе доказательства основного утверждения, нам будет необходима следующая теорема, принадлежащая Ю.~И.~Журавлеву \cite{zh58}
\begin{theorema}\label{caption_criterion}
Конъюнкция $K$ поглощается совокупностью конъюнкций {$K_1,...,K_t$}~{,}~если для конъюнкций $K_1', ..., K_t'$~{,}~которые получены вычеркиванием литералов, содержащихся в $K$ из $K_1,...,K_t$ соответственно, выполнено
\[K_1' \vee K_2' \vee ... \vee K_t' \equiv 1\]
если при этом, конъюнкция, из которой были вычеркнуты все литералы, полагается тождественно равной единице.
\end{theorema}

\begin{ddef}
\textit{Разрезанием} матрицы $M$, являющейся тестом, на $t, t\ge1$ частей, назовем совокупность подматриц $\{M_1\}_{i=1}^t$ таких, что
\begin{enumerate}
\item Матрицы $M_1, M_2,...,M_t$ имеют одинаковое число столбцов;
\item каждая строка матрицы $M$ присутствует хотя бы в одной матрице $M_j, 1\le j \le t$;
\item число строк матрицы $M$ совпадает с суммой числа строк в матрицах $\{M_i\}_{i=1}^t$.
\end{enumerate}
\end{ddef}

\begin{lemma}\label{literal_1}
Всякий литерал функции $f$ входит не менее чем в одну конъюнкцию.
\end{lemma}
\begin{lemma}\label{literal_coverage}
Пусть литерал $x_i^{\sigma_i}$ разложим по литералам $x_{i_1}^{\sigma_{i_1}\oplus1}, x_{i_2}^{\sigma_{i_2}\oplus1},...$, $x_{i_t}^{\sigma_{i_t}\oplus1}$.
Пусть $\beta = (\beta_1,\beta_2,...,\beta_n)$ --- такая точка булева куба, что $\beta_i = \sigma_i, \beta_{i_1} = \sigma_{i_1}, \beta_{i_2} = \sigma_{i_2},..., \beta_{i_t} = \sigma_{i_t}$.
Тогда конъюнкция $K=x_i^{\sigma_i}x_{i_1}^{\sigma_{i_1}} x_{i_2}^{\sigma_{i_2}}\cdots x_{i_t}^{\sigma_{i_t}}$ является допустимой для $f$ и, более того, $K[\beta]~=~1$
\end{lemma}

Конъюнкцию $K$ в условиях леммы назовем конъюнкцией определяемой разложением $x_i^{\sigma_i}$ по векторам $x_{i_1}^{\sigma_{i_1}\oplus1}, x_{i_2}^{\sigma_{i_2}\oplus1},...,x_{i_t}^{\sigma_{i_t}\oplus1}$. Докажем следующую важную \textit{лемму о минимальном числе литералов}

\begin{lemma}\label{literal_min_number}
Рассмотрим приведенную булеву функцию $f(x_1,...,x_n)$ заданную матрицей нулей $M_f$ и не имеющую смежных нулевых точек. Зафиксируем ее литерал $x_i^{\sigma_i}, \sigma_i \in\{0,1\}$. Справедливы следующие утверждения
\begin{enumerate}
\item[$a.$] \textbf{Случай фиксированной ДНФ.} Пусть $D$ --- некоторая ДНФ функции $f$. Выделим в $D$ конъюнкции, содержащие литерал $x_i^{\sigma_i}$ в количестве $t$ штук. Литерал $x_i^{\sigma_i}$ входит еще как минимум в одну, отличную от выделенных, конъюнкцию из $D$, если при любом разрезании матрицы $M_f$ на $t$ частей $M_1,M_2,...,M_t$ существует такая матрица $M_j$, что для функции $\phi_j$ заданной матрицей $M_j$, ни одна из выделенных конъюнкций не определяет разбиение $x_i^{\sigma_i}$.
\item[$b.$] \textbf{Равномерный случай.} Литерал $x_i^{\sigma_i}$ входит не менее чем в $t+1$ конъюнкцию каждой ДНФ функции $f$, если при любом разрезании матрицы $M_f$ на $t$ частей $M_1,M_2,...,M_t$, существует такая матрица $M_j$, что для функции $\phi_j$, заданной матрицей $M_j$, не существует набора литералов, образующего разбиение литерала $x_i^{\sigma_i}$.
\item[$c.$] Кроме того, если при разрезании матрицы $M$ на части $M_1,M_2,...$, $M_t$ для каждой из функций, определяемых матрицами $M_1,M_2,...$, $M_t$, существует (свое для каждой части) разбиение литерала $x_i^{\sigma_i}$, порождающее при этом допустимую конъюнкцию $f(x_1,x_2,...,x_n)$, то существует ДНФ $\hat D$ функции $f$, такое что литерал $x_i^{\sigma_i}$ входит не более чем в $t$ различных конъюнкций содержащихся в $\hat D$.
\end{enumerate}
\end{lemma}
\begin{lemm_conseq}
Литерал $x_i^{\sigma_i}$ входит не менее чем в две конъюнкции произвольной ДНФ~$D$, если отрицания литералов, содержащихся в конъюнкции $K\setminus x_i^{\sigma_i}$, при $K\ni x_i^{\sigma_i}$, не образуют разбиение~$x_i^{\sigma_i}$.
\end{lemm_conseq}
\begin{proof}
Отметим, что утверждение $b$ (равномерный случай) является следствием утверждения леммы в случае конкретной ДНФ. В самом деле, из существования разложений, по лемме \ref{literal_coverage} следует существование конъюнкций, основанных на указанных разложениях, и ДНФ, содержащей эти конъюнкции.

Докажем следствие, не опираясь на утверждение леммы. Согласно лемме \ref{literal_1}, литерал $x_i^{\sigma_i}$ входит как минимум в одну конъюнкцию $K=x_i^{\sigma_i}x_{i_1}^{\sigma_{i_1}}x_{i_2}^{\sigma_{i_2}}\cdots x_{i_t}^{\sigma_{i_t}}$ из $D$. Пусть отрицания литералов из $K\setminus x_i^{\sigma_i}$ не образуют разбиение $x_i^{\sigma_i}$. Каждую строку подматрицы $M_{[k]}^{i,i_1,...,i_t}$ матрицы нулей $M$ функции $f$ сложим поэлементно со строкой $(\sigma_i\oplus1, \sigma_{i_1}, \sigma_{i_2}, ..., \sigma_{i_t})$. Полученную матрицу обозначим $\hat M$. Так как отрицания литералов из $K\setminus x_i^{\sigma_i}$ не образуют разбиение $x_i^{\sigma_i}$, то в полученной матрице есть либо строка состоящая только из нулей, либо строка, начинающая с единицы и содержащая по крайней мере 2 единицы. Рассмотрим каждый случай отдельно:
\begin{itemize}
\item[$a.$] В матрице $\hat M$ есть нулевая строка. В этом случае в подматрице $M_{[k]}^{i,i_1,...,i_t}$ матрицы нулей $f$ существует подстрока $(\sigma_i, \sigma_{i_1}, \sigma_{i_2}, ..., \sigma_{i_t})$, следовательно конъюнкция $K$ покрывает один из нулей функции $f$ и не может входить в ее ДНФ.
\item[$b.$] Матрица $\hat M$ не содержит нулевых строк. В этом случае в $\hat M$ содержится строка с номером $r$, начинающаяся с единицы и содержащая по крайней мере две единицы.

    Обозначим $\alpha = (\alpha_1,\alpha_2,...,\alpha_n)$ строку матрицы $M$ с указанным номером $r$; $\beta$ --- булеву строку, отличающуюся от $\alpha$ только в разряде с номером $i$; $i_j$ номер столбца $M$, такой что $\alpha_{i_j} \oplus 1 = \sigma_{i_j}$. Заметим, что $\alpha_i = \sigma_i$.

    В силу отсутствия у $f$ соседних нулевых точек,~$f(\beta) = 1$.

    Пусть $K_0 = x_1^{\alpha_{1}} x_2^{\alpha_{2}}\cdots  x_{i-1}^{\alpha_{i-1}} x_{i+1}^{\alpha_{i+1}} \cdots  x_{i_j-1}^{\alpha_{i_j-1}} x_{i_j+1}^{\alpha_{i_j+1}} \cdots x_n^{\alpha_{n}}$; $K_1 =  x_i^{\alpha_i}x_{i_j}^{\alpha_{i_j}}$; $K_0 = x_i^{\sigma_{i}} x_{i_j}^{\sigma_{i_j}\oplus 1}K_0$; $K_2 = x_i^{\alpha_i\oplus1}x_{i_j}^{\alpha_{i_j}} K_0 = x_i^{\sigma_{i}\oplus 1} x_{i_j}^{\sigma_{i_j}\oplus 1} K_0$ при этом $K_1 \not\in D$, а $K_2 \in D$. Положим $D'$ --- ДНФ, получающаяся из $D$ вычеркиванием всех литералов содержащихся в $K_0$.

    Воспользуемся критерием поглощения. Так как $\alpha$ является нулем $f$, то $D'$ не может быть тождественно истинной. В случае, когда $K$ --- единственная конъюнкция содержащая $x_i^{\sigma_i}$ выполнено $x_i^{\sigma_i} x_{i_j}^{\sigma_{i_j}} \subseteq D'$, то есть $D' \subseteq x_i^{\sigma_i\oplus 1} \vee x_i^{\sigma_i} x_{i_j}^{\sigma_{i_j}}$; если $K_2 \in D$, то после вычеркивания литералов $x_i^{\sigma_i}$ и $x_{i_j}^{\sigma_{i_j}\oplus 1}$ из $D'$, последняя должна обратиться в тождественную единицу. Это неверно, а именно, конъюнкция $K_2$ являясь импликантой $f$, не покрывается ДНФ $D$ следовательно в рассматриваемом случае $K$ не может быть единственной конъюнкцией $D$ содержащей $x_i^{\sigma_i}$. Следствие доказано.
\end{itemize}

Доказательство основного случая утверждения $a$ проводится аналогично. Для конъюнкции $K_1$ выделим максимальный набор строк такой, что для образованной ими матрицы $M_1$ отрицания литералов $K_1$ образуют разбиение $x_i^{\sigma_i}$. Из оставшихся не выделенными строк $M$, выделим максимальный набор строк такой, что для образованной ими матрицы $M_2$ отрицания литералов $K_2$ образует разбиение $x_i^{\sigma_i}$. Аналогично поступим для конъюнкций $K_3, K_4,..., K_{t-1}$.

Обозначим $M_t$ --- матрицу, состоящую из оставшихся не выделенными строк. По условию леммы, отрицания литералов $K_t\setminus x_i^{\sigma_i}$ не образуют разбиения $x_i^{\sigma_i}$ в $M_t$. По доказанному ранее следствию существуют такие строки $\alpha \in M_t$ и $\beta$, что $\alpha$ --- ноль функции $f$, а $\beta$ --- соседняя с $\alpha$ точка, не покрываемая конъюнкцией $K_t$, такая что $f(\beta) = 1$.

Рассмотрим матрицы $M_1',M_2',...,M_{t-1}'$, получаемые из $M_1,M_2,...,M_{t-1}$ добавлением строки~$\alpha$. Заметим, что в силу максимальности матриц $M_1,M_2,...,M_{t-1}$ отрицание литералов конъюнкции $K_j$ не может определять разложение литерала $x_i^{\sigma_i}$ в матрицах $M_j', 1\le j\le t-1$, при этом всякий раз точка $\beta$ оказывается непокрытой. Следовательно для ДНФ $D$ функции $f$ в условиях леммы литерал $x_i^{\sigma_i}$ входит как минимум в $t+1$ конъюнкцию.

Справедливость утверждения $c$ можно показать следующим образом. Пусть $D_0$ --- ДНФ функции $f_0(x_1,x_2,...,x_{i-1},x_{i+1},...,x_n)$, определяемой матрицей нулей $M_0 = M_{[k]}^{1,2,...,i-1,i+1,...,n}$. Так как $f(x_1,x_2,...,x_n)$ --- булева функция без смежных нулей, то $M_0$ тест.

Пусть $x_i^{\sigma_i}K_1,x_i^{\sigma_i}K_2,..., x_i^{\sigma_i}K_t$ --- конъюнкции, порождаемые разложениями $x_i^{\sigma_i}$. Пусть $\sigma(i) = \{\alpha_{r_1},\alpha_{r_2},..., \alpha_{r_s}\}$ --- множество, содержащее все такие строки матрицы $M$, что $M_{r_m}^{i} = \sigma_i, M_{r_m}^{[n]} = \alpha_{r_m}, 1\le m\le s$.

Рассмотрим следующую ДНФ
\[
D_f = D_0 \vee x_i^{\sigma_i\oplus1}\wedge\left(\bigvee\limits_{(\gamma_1,\gamma_2,...,\gamma_n) \in \sigma(j)} x_1^{\gamma_1} x_2^{\gamma_2}
\cdots x_{i-1}^{\gamma_{i-1}} x_{i+1}^{\gamma_{i+1}} \cdots x_n^{\gamma_n}\right) \vee x_i^{\sigma_i}\wedge\left(\bigvee\limits_{p=1}^t K_p\right)
\]

Покажем, что указанная ДНФ реализует $f(x_1,x_2,...,x_n)$. Пусть $\beta = (\beta_1,\beta_2,...,\beta_n)$ --- произвольная ненулевая точка $f$.

Если вектор $(\beta_1,\beta_2,...,\beta_{i-1},\beta_{i+1},...,\beta_n)$ отличен от нуля функции $f_0$, то $D_0[\beta]~=~1$.

Пусть теперь $\beta = \alpha \oplus e_i^n$, где точка $\alpha$ является нулем функции $f$. Если $\alpha \in \sigma(i)$, то $\beta$ покрывается одной из конъюнкций содержащих $x_i^{\sigma_i\oplus1}$. В противном случае рассмотрим существующее по условию леммы разрезание матрицы $M$ на части $M_1, M_2,..., M_t$, такое что каждая из указанных частей образует разбиение литерала $x_i^{\sigma_i}$  и, более того, конъюнкции, определяемые этими разбиениями, допустимы для $f$.

Пусть $M_j$ часть матрицы нулей содержащая $\alpha$, $x_i^{\sigma_i}K_j$ --- допустимая конъюнкция, порожденная $M_j$. Точка $\beta$ покрывается $x_i^{\sigma_i}K_j$ по лемме \ref{literal_coverage}. Лемма доказана.
\end{proof}
\begin{lemma}\label{literal_chi}
Пусть $\alpha \in B_k$ разложима по векторам $\{\alpha_i\}_{i=1}^t, \alpha_i \in B_k$, $i~\in~[1,t]$. Тогда справедливо неравенство
\[\chi(\alpha_1) + ... + \chi(\alpha_t) \ge \chi(\alpha)\]
\end{lemma}
\begin{ddef}
Литерал $x_i^{\sigma_i}$ назовем \textit{собственным} литералом конъюнкции $K$ ДНФ $D$ функции $f(x_1,x_2,...,x_n)$, если $K$ --- единственная конъюнкция $D$ содержащая $x_i^{\sigma_i}$.
\end{ddef}

Для доказательства основной теоремы о сложности полной функции нам потребуется следующая \textit{лемма о свойствах полной функции}
\begin{lemma}\label{lemma_complete_properties}
Для полной булевой функции $G_k(x_1, x_2, ..., x_n)$ выполнено
\begin{enumerate}
\item Если конъюнкция $K \in D$ антимонотонна, то $rank\; K \ge 3$;
\item Если $K$ --- единственная конъюнкция, содержащая $x_i^{\sigma_i}$, то $rank\; K \ge 3$;
\item Разложение литерала $x_i$ не может содержать литерал $\bar{x}_j$, $1\le~i,j\le~n, i\neq~j$;
\item Разложение $\bar{x}_i\wedge x_j$ не может содержать литералов веса больше либо равного $k/4$, если вес $x_i$ больше чем $k/4$, а литералы $\bar{x}_i$ и $x_j$ ортогональны;
\item Разложение $\bar{x}_i\wedge\bar{x}_j$ не может содержать литералов веса больше либо равного $k/3$, если веса $\bar{x}_i$ и $\bar{x}_i$ больше чем $k/3$, а литералы $x_i$ и $x_j$ ортогональны;
\end{enumerate}
\end{lemma}
\begin{proof}
По определению функции $G_k$, число единиц в каждом столбце ее матрицы не превосходит числа единиц. Функция $G_k$ приведенная, следовательно не существует пары литералов, образующих ортогональное разложение $\mathbb{I}_k$, и утверждения 1 и 2 справедливы.

Кроме того, $ind(x_i) \le ind(x_j), 1\le i,j \le n$. Так как $G_k$ приведенная, указанное неравенство доказывает утверждение 3.
Утверждения 4 и 5 являются прямыми следствиями леммы \ref{literal_chi}.
\end{proof}

В работе \cite{dj01} А.Г. Дьяконовым доказаны следующие две леммы для длины полной и приведенной функций от $n$ переменных:
\begin{lemma}\label{lemm_dj_1}
Длина всякой ДНФ полной функции $f$ от $n$ переменных не меньше $n$.
\end{lemma}
\begin{lemma}\label{lemm_dj_2}
Длина всякой приведенной функции $f$ от $n$ переменных, имеющей $k$ нулей, удовлетворяет неравенству
\begin{gather*}
|D| \le 2n + \left(k^2 - 5k\right)/2
\end{gather*}
в случае, если матрица нулей функции $f$ содержит единичную подматрицу с точностью до перестановок и инверсий столбцов.
\end{lemma}


Обозначим $\mathbb{M}$ множество литералов веса большего чем $k/3$.
\begin{lemma}\label{lemm_norm_card}
Число конъюнкций произвольной ДНФ $D$ полной функции $f$, содержащих литералы из $\mathbb{M}$, не меньше $n\left(1 - e^{-\log_2n/36}\right)$.
\end{lemma}
\begin{proof}
Выделим подматрицу $M'$ матрицы нулей полной функции, содержащую все столбцы веса, не превосходящего $k/3$. По неравенству Чернова число столбцов в $M'$ не превосходит $2^{k-1} e^{-k/36}$.

По лемме \ref{lemm_dj_2} существует ДНФ $D_0$, реализующая функцию заданную матрицей нулей $M'$, длина которой не превосходит
\[|D_0| \le 2t + \frac{k^2 - 5k}{2} < k_0 \cdot n = 6\cdot ne^{-\log_2n/36}
\]

Таким образом, если существует такая ДНФ $D_1$ функции $G_k$, что число конъюнкций, содержащих переменные из $\mathbb{M}$, меньше $n(1-k_0)$, то существует и ДНФ $D_2$ функции $G_k$, полученная из $D_1$ заменой всех конъюнкций, не содержащих переменные из $\mathbb{M}$, на ДНФ $D_1$. При этом длина ДНФ $D_2$, которая реализует $G_k$, меньше $n$. Последнее противоречит лемме \ref{lemm_dj_2}, поэтому число конъюнкций, содержащих литералы из $\mathbb{M}$, не меньше $n\left(1 - e^{-\log_2n/36}\right)$.
\end{proof}
\begin{theorema}\label{bound_lower}
Минимальная ДНФ полной булевой функции от $n$ переменных имеет ранг не меньший чем $3n\left(1 - e^{-\log_2n/36}\right)$.
\end{theorema}
\begin{proof}
Рассмотрим произвольную ДНФ $D$ функции $G_k$. Выделим в ней все конъюнкции, содержащие литералы из $\mathbb{M}$, и ограничимся далее их рассмотрением. Число $m$ таких конъюнкций по лемме \ref{lemm_norm_card} не меньше $n(1-e^{-\log_2n/36})$.

Пусть $k_1$ --- число конъюнкций, имеющих один собственный литерал; $k_2$ --- число конъюнкций, имеющих два собственных литерала. Отметим, что никакие 3 собственных литерала из $\mathbb{M}$ не могут входить в одну импликанту ДНФ $D$ по лемме \ref{lemma_complete_properties}.

Следовательно, число входящих в $D$ несобственных литералов(без учета кратности), не меньше $2|\mathbb{M}| - k_1 - k_2$. Из леммы \ref{lemma_complete_properties}(утверждение 2) следует, что если конъюнкция $K$ содержит единственный собственный литерал, то в $K$ содержится еще как минимум два несобственных литерала, откуда совокупная кратность всех несобственных литералов не меньше $2k_1$; если же конъюнкция $K$ содержит два собственных литерала, то $rank\;K\ge 3$ и как минимум один из них не содержится в $\mathbb{M}$.

Откуда имеем следующую оценку для ранга:
\[rank\; D \ge \max(2(2|\mathbb{M}| - 2k_2 - k_1), 2k_1) + k_1 + 3k_2.\]

Откуда при $k_1+k_2 \ge \left|\mathbb{M}\right|$ получаем $rank\;D\ge 4|\mathbb{M}| - k_1 - k_2 \ge 3|\mathbb{M}|$.
Кроме того, при $k_1 + k_2 \le |\mathbb{M}|$ имеем $rank\; D \ge 3k_1+3k_2 \ge 3|\mathbb{M}|$.

Из того, что $|\mathbb{M}| \ge n\left(1 - e^{-\log_2n/36}\right)$, при $n\rightarrow\infty$ получаем утверждение теоремы.
\end{proof}

\section{Алгоритм построения ДНФ~полной~функции}
В данном разделе предлагается эффективный метод построения ДНФ, содержащих небольшое число литералов, для булевых функций с малым числом нулей. В частности указанный метод позволят строить асимптотически минимальные ДНФ для \textit{полной булевой функции}.

В работе \cite{dj02} А.~Г.~Дьяконовым был предложен достаточно универсальный тестовый подход к построению ДНФ булевых функций. Основная идея подхода состоит в сведении построения ДНФ исходной функции $f(x_1,x_2...,x_n)$ к задаче построения ДНФ булевой функции $\phi(y_1,y_2,...,y_t)$, заданной матрицей нулей с меньшим числом столбцов такой, что матрица нулей $\phi$ является подматрицей матрицы нулей $f$ и, более того, является тестом. Однако сложность такого сведения в общем случае чрезвычайно высока. А именно, ДНФ $D_f$ может быть получена из ДНФ $D_{\phi}$ добавлением $(n-t)k$ конъюнкций, содержащих $(n-t)k(t+1)$ литералов, где число нулей функций $f$ и $\phi$ равно $k$.

В работе \cite{dj01} А.~Г.~Дьяконов рассматривал существенно более частный класс функций. А именно, функции матрицы нулей которых содержат единичную подматрицу максимального размера. Для указанных функций им были предложены алгоритмы, гарантирующие лучшие оценки сложности, чем в общем случае. Приведенные в работе \cite{dj01} алгоритмы позволили построить ДНФ полной функции, содержащей $n(1+o(1))$ конъюнкций и $\Theta(n\sqrt{\log_2 n})$ литералов.

Предлагаемый в настоящей работе для построения ДНФ алгоритм основан на конструкциях А.~Г.~Дьяконова. Однако для рассматриваемых далее классов булевых функций он обладают существенно более низкими оценками сложности по сравнению с алгоритмами, предложенными в \cite{dj02}. В частности, методы, представленные в работе, позволяют строить ДНФ полной функции, содержащую $n(1+o(1))$ конъюнкций и $3n(1+o(1))$ литералов.

Зафиксируем приведенную функцию $f(x_1,x_2,...,x_n)$ от $n$ переменных, имеющую $k$ нулевых точек и заданную матрицей нулей $M_f$. Пусть существует подматрица $T$ матрицы $M_f$, содержащая $t$ столбцов; без ограничения общности положим, что столбцы матрицы $T$ соответствуют столбцам с номерами $1,2,...,t$ матрицы $M_f$, так как в противном случае инвариантными преобразованиями Шеннона-Поварова функция $f(x_1,x_2,...,x_n)$ может быть преобразована к указанному виду. Пусть столбцы матрицы $M_f$, не входящие в $T$, образуют матрицу~$M'$. Литералы, ассоциированные со столбцами $T$, назовем \textit{тестовыми}, все остальные литералы назовем \textit{внешними}.

\begin{ddef}
\textit{Гиперграфом связей} $H_T := H_T[f(x_1,x_2,...,x_n)]$ функции $f(x_1,x_2,...,x_n)$~с~матрицей нулей $M_f$ и тестовой подматрицей $T$ назовем пару состоящую из гиперграфа и взаимно-однозначного отображения $\phi$ между литералами функции $f$ и вершинами гиперграфа, при котором
\begin{enumerate}
\item в гиперграфе присутствует 2-ребро если и только если вершинам ребра соответствует пара литералов $\{x_i, \bar{x}_i\},1\le i\le n$;
\item если и только если набор из $t$ литералов образует (ортогональное) разбиение $\mathbb{I}_k$, гиперграф содержит отвечающее ему $t$-ребро.
\end{enumerate}
\end{ddef}

Множество вершин гиперграфа $H$ обозначим $V(H)$, а множество его ребер --- $E(H)$. Далее под вершиной $x_i^{\sigma_i}$ будет подразумеваться вершина, являющаяся образом литерала $x_i^{\sigma_i}$ при отображении $\phi$. Вершины гиперграфа, которым соответствуют внешние(тестовые) литералы, назовем внешними(тестовыми) вершинами.

\begin{ddef}\label{hypergraph_path}
Пусть задан гиперграф связей $H_T[f(x_1,x_2,...,x_n)]$ функции $f(x_1,x_2,...,x_n)$ с матрицей нулей $M_f$ и тестовой подматрицей $T$.
\textit{Путем} между внешними литералами $x_i^{\sigma_i}$ и $x_j^{\sigma_j}$ в $H_T$, назовем последовательность переходов между внешними литералами внутри ребер гиперграфа, начинающуюся в $x_i^{\sigma_i}$ и заканчивающуюся $x_j^{\sigma_j}$.
\end{ddef}

Гиперграфы, полученные из гиперграфа связей удалением одного или нескольких ребер, назовем \textit{частичными}. Определение пути, данное для гиперграфа связей, естественным образом обобщается на все частичные гиперграфы. Как правило, нас будут интересовать только частичные гиперграфы связей с небольшим числом ребер.

\begin{figure}[tH]
\noindent\centering{
\includegraphics[scale=0.25]{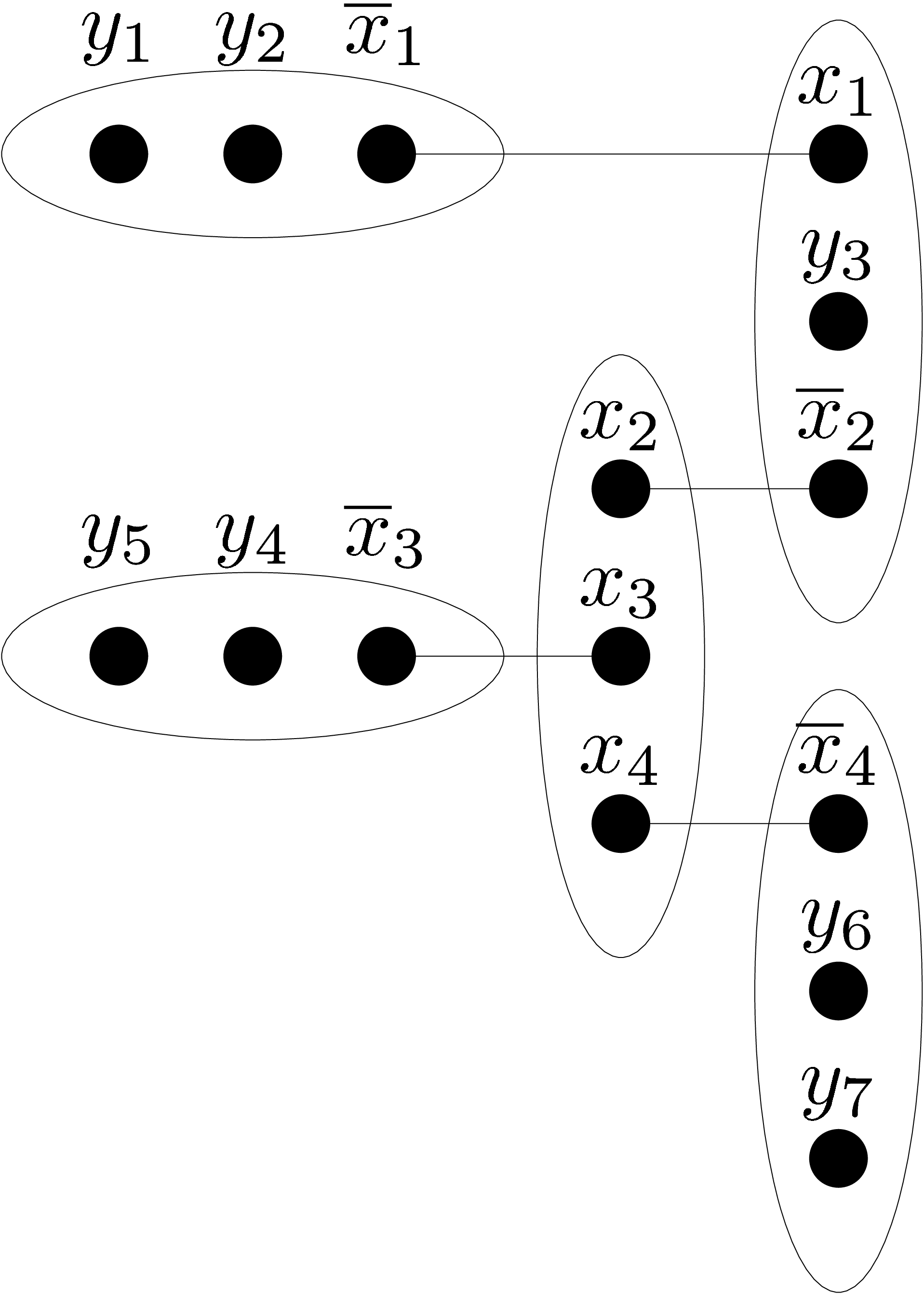}
}
\caption{Гиперграф связей из примера \ref{example_graph}.}
\label{pic_hypergraph}
\end{figure}
\begin{example}\label{example_graph}
Рассмотрим матрицу $M$, содержащую тестовую подматрицу $T$ в первых семи столбцах.
\[M = \left(
  \begin{array}{ccccccccccc}
    1 & 0 & 0 & 0 & 1 & 0 & 0 & 0 & 0 & 1 & 0 \\
    0 & 1 & 0 & 1 & 0 & 0 & 0 & 0 & 0 & 1 & 0 \\
    0 & 1 & 0 & 0 & 0 & 1 & 0 & 0 & 0 & 0 & 1 \\
    0 & 0 & 1 & 0 & 0 & 0 & 0 & 0 & 1 & 0 & 0 \\
    0 & 0 & 0 & 0 & 0 & 0 & 0 & 1 & 1 & 0 & 0 \\
    1 & 0 & 0 & 0 & 0 & 0 & 1 & 0 & 0 & 0 & 1 \\
  \end{array}
\right)\]
Матрица $M$ определяет функцию $f(y_1,y_2,...,y_7,x_1,x_2,x_3,x_4)$, литералы $y_1$ -- $y_7$ которой ассоциированы с $T$, а литералы $x_1$ -- $x_4$ с $M\setminus T$. На рисунке \ref{pic_hypergraph} изображен один из частичных гиперграфов связи функции $f$.
\end{example}

Назовем ребро гиперграфа \textit{корнем}, если оно содержит ровно один внешний литерал. Ребра гиперграфа назовем смежными, если они имеют хотя бы одну общую вершину. Размером ребра назовем число вершин в нем.

\begin{ddef}
Для приведенной функции $f(x_1,x_2,...,x_n)$ рассмотрим гиперграф $S_T$, полученный удалением из гиперграфа связей $H_T = H_T[f(x_1,x_2,...,x_n)]$ некоторых вершин и ребер. Гиперграф $S_T$ назовем \textit{последовательностью} (ортогональных) разложений $\mathbb{I}_k$, если
\begin{enumerate}
\item Всякий внешний литерал $x_i^{\sigma_i}$, являющийся вершиной $S_T$, входит ровно в два ребра, одним из которых является ребро $\{x_i, \bar{x}_i\}$;
\item Всякая внешняя вершина графа лежит хотя в одном пути, начало и конец которого являются корневыми;
\item смежными для ребер размера 2, являются ребра размера 3 и более, либо корневые ребра;
\end{enumerate}
\end{ddef}


\textit{Мощностью} последовательности ортогональных разложений $S_T$, построенной по тесту $T$, будем называть половину от числа вершин $S_T$, ассоциированных с внешними литералами гиперграфа. На рисунке \ref{pic_hypergraph} приведена одна из таких последовательностей мощности 4, построенная для определенных в примере \ref{example_graph} теста $T$ и функции $f$.

Рассмотрим ребро $e_j = \{x_{i_1}^{\sigma_{i_1}}, x_{i_2}^{\sigma_{i_2}}, ..., x_{i_t}^{\sigma_{i_t}}\}$ гиперграфа $S_T$. Литералы $x_{i_1}^{\sigma_{i_1}}, x_{i_2}^{\sigma_{i_2}}, ..., x_{i_t}^{\sigma_{i_t}}$ образуют ортогональное разбиение единицы(для матрицы $M_f$), следовательно, если $e_j$ не является ребром, соединяющим пару противоположных литералов, то $K[e_j] = x_{i_1}^{\sigma_{i_1}\oplus1} x_{i_2}^{\sigma_{i_2}\oplus1}\cdots x_{i_t}^{\sigma_{i_t}\oplus1}$ простая импликанта $f(x_1,x_2,...,x_n)$ .

\begin{lemma}\label{hyperspan_elementary_coverage}
Пусть для приведенной функции $f(x_1,x_2,...,x_n)$ и теста $T$ существует последовательность ортогональных разложений $S_T$, содержащая все внешние литералы.
Тогда функция $f$ реализуется ДНФ $D_{S_T}$, определяемой равенством
\[D_{S_T} = D_T \vee \bigvee\limits_{e\in E(S_T)} K[e]\]
где $K[e] = x_{i_1}^{\sigma_{i_1}\oplus1} x_{i_2}^{\sigma_{i_2}\oplus1}\cdots x_{i_t}^{\sigma_{i_t}\oplus1}$ при $e = \{x_{i_1}^{\sigma_{i_1}}, x_{i_2}^{\sigma_{i_2}},..., x_{i_t}^{\sigma_{i_t}}\}$.
\end{lemma}
\begin{proof}
Рассмотрим функцию $p$, определенную на парах состоящих из ребер гиперграфа $C_T$ и точек булева куба $B_n$. Определим значение $p$ для ребра $e = \{x_{i_1}^{\sigma_{i_1}}, x_{i_2}^{\sigma_{i_2}},..., x_{i_t}^{\sigma_{i_t}}\}$ и точки $\beta = (\beta_1,\beta_2,...,\beta_n)$ как число выполненных равенств $\beta_{i_j} = \sigma_{i_j}$. Заметим, что если $p(e,\beta) = 0$, точка $\beta$ покрывается конъюнкцией $K[e]$, а если $p(e,\beta)=1$ при всех $e\in E(S_T)$, то
$\beta$ является нулем $f$ или покрывается ДНФ $D_T$.

Сумма $\sum\limits_{e \in E(S_T)} p(e, \beta)$ является инвариантом для всех точек $\beta$, совпадающих с одним из нулей функции в тестовых столбцах, и равна числу ребер $S_T$, порождающих отличные от нуля конъюнкции. Для доказательства утверждения леммы остается показать, что не существует точки $\beta \in B_n$, отличной от перечисленных в $M_f$, для которых выполнены равенства $p(e,\beta) = 1$, при всех $e\in E(S_T)$.

Докажем утверждение индукцией по числу столбцов $m$, принадлежащих матрице $M_f$, но не принадлежащих $T$. Для удобства дальнейших рассуждений положим, что столбцы с номерами $1,2,...,m$ не принадлежат $T$.

Пусть $m=1$. В этом случае всякая последовательность ортогональных разложений $\mathbb{I}_k$ содержит как минимум 3 ребра, тестовые литералы первого из которых образуют ортогональное разложение $x_i$, литералы второго --- ортогональное разложение $\bar{x}_i$, а третье ребро $\{x_i,\bar{x}_i\}$. Указанные ребра для каркаса $C_T$ обозначим $e_1,e_2$ и $e_3$ соответственно. Осталось заметить, что всякая точка булева куба, не совпадающая ни с одним из нулей функции в тестовых столбцах, покрывается ДНФ $D_T$; всякая же точка $B_k$, отличающая от нуля $f$ только в первой координате, покрывается $K[e_1]$ или $K[e_2]$.

Пусть $m = t+1, t\ge1$. Пусть $\beta\in B_n$ отлична от нулей $f$ и не покрывается $D_T$. Следовательно точка $\beta$ совпадает с одним из нулей функции, перечисленных в $M_f$, в тестовых столбцах, обозначим указанный ноль $\alpha$. Пусть $\alpha$ и $\beta$ отличаются в разряде $i$. Литералы $x_i$ и $\bar{x}_i$ внешние. Обозначим $e_0 = \{x_i,\bar{x}_i\}$; $x_i\in e_1, e_1\neq e_0$; $x_i\in e_2, e_2\neq e_0$. По построению $S_T$ не имеет иных ребер, содержащих литералы $x_i$ и $\bar{x}_i$.

Пусть $\beta_i = 0 \Rightarrow \alpha_i = 1$(Случай $\beta_i = 1$ рассматривается аналогично). Откуда $p(e_1, \beta) = 1$ только в случае, если существует внешний литерал $x_j^{\sigma_j} \in e_1: \beta_j = \sigma_j$. По определению $S_T$ литерал $x_j^{\sigma_j}$ входит только в 2 ребра, а именно, в ребро $e_1$ и в ребро $\{x_j, \bar{x}_j\}$; ребро, в которое входит литерал $x_j^{\sigma_j\oplus1}$, обозначим $e_3$. Заметим, что $p(e_3,\alpha) = 1$, а $p(e_3\setminus x_j^{\sigma_j\oplus1}, 0)$. Следовательно, точка $\beta$ может быть нулем функции только в том случае, если существует внешний литерал $x_r^{\sigma_r} \in e_3: \beta_r = \sigma_r$. Так как удаление ребра $\{x_j\, \bar{x}_j\}$ разбивает гиперграф на пару непересекающихся по внешним вершинам гиперграфов, анализ сведен к гиперграфу с меньшим числом внешних вершин. Лемма доказана.
\end{proof}
\begin{ddef}\label{hypergraph_span}
Для приведенной функции $f(x_1,x_2,...,x_n)$ и теста $T$ рассмотрим множество $\bigcup\limits_{i=1}^t S^i_T$. Частичный гиперграф $C_T$ гиперграфа связей  $H_T[f(x_1,x_2,...,x_n)]$, полученный из $\bigcup\limits_{i=1}^t S^i_T$ удалением кратных ребер назовем \textit{каркасом, построенным по тесту $T$}, если $C_T$ содержит все внешние вершины.
\end{ddef}

Число ребер каркаса, построенного по последовательностям ортогональных разложений, может быть существенно меньше, чем сумма числа ребер входящих в последовательности разложений. Указанный эффект имеет место, когда последовательности, имеют большое число общих ребер. Этот эффект является мотивацией для использования в дальнейших рассуждениях каркасов, а не последовательностей ортогональных разложений.

\begin{lemma}\label{hyperspan_coverage}
Рассмотрим приведенную булеву функцию $f(x_1,x_2,...,x_n)$, заданную матрицей нулей $M_f$ с выделенным тестом $T$. Обозначим $D_T$ некоторую ДНФ булевой функции, определяемой матрицей нулей $T$. Пусть, кроме того, существует каркас $C_T[f(x_1,x_2,...,x_n)]$, содержащий все внешние литералы. Тогда функция $f$ реализуется ДНФ $D$, которая определяется равенством
\[D = D_T \vee \sum_{e \in E(C_T)} K[e]\]
где $K[e] = x_{i_1}^{\sigma_{i_1}\oplus1} x_{i_2}^{\sigma_{i_2}\oplus1}\cdots x_{i_t}^{\sigma_{i_t}\oplus1}$ при $e = \{x_{i_1}^{\sigma_{i_1}}, x_{i_2}^{\sigma_{i_2}},..., x_{i_t}^{\sigma_{i_t}}\}$.
\end{lemma}
\begin{proof}
Заметим, что все возможные нули функции совпадают со строками матрицы нулей в тестовых столбцах. По построению, для каждой конъюнкции из $D\setminus D_T$ не выполнен ровно один входящий в нее литерал, следовательно все нули функции $f(x_1,x_2,...,x_n)$, определяемые матрицей $M_f$, остаются непокрытыми.

Пусть $\beta$ некоторый ноль $D$, не являющийся нулем $f$. Пусть $\alpha \in M_f$ такой, что $\alpha$ и $\beta$ совпадают в тестовых столбцах. Кроме того, пусть $i:\alpha_i\neq\beta_i$. Выделим из $C_T$ гиперграф $S_T$, содержащий $x_i$ и $\neq{x}_i$. Воспользуемся утверждением леммы \ref{hyperspan_elementary_coverage}.

\end{proof}

Следствием доказанного утверждения является лемма
\begin{lemma}\label{hypergraph_span_union}
Рассмотрим приведенную булеву функцию $f(x_1,x_2,...,x_n)$, заданную матрицей нулей $M_f$ с выделенным тестом $T$. Обозначим $D_T$ некоторую ДНФ булевой функции, определяемой матрицей нулей $T$. Пусть, кроме того, существует некоторое множество каркасов $\mathbb{C}$, обладающих общим тестом $T$ и содержащих в совокупности все внешние литералы. Тогда функция $f$ реализуется ДНФ $D$, которая определяется равенством
\[D = D_T \vee \sum\limits_{C_T \in\mathbb{C}} \quad\sum_{e \in E(C_T)} K[e]\]
где $K[e] = x_{i_1}^{\sigma_{i_1}\oplus1} x_{i_2}^{\sigma_{i_2}\oplus1}\cdots x_{i_t}^{\sigma_{i_t}\oplus1}$ при $e = \{x_{i_1}^{\sigma_{i_1}}, x_{i_2}^{\sigma_{i_2}},..., x_{i_t}^{\sigma_{i_t}}\}$, .
\end{lemma}

Указанные леммы дают эффективные алгоритмы построения ДНФ, когда тест и каркасы могут быть построены эффективно. Это характерно для булевых функций высокой размерности с небольшим числом нулей. Одним из примеров указанных лемм на практике является построение ДНФ \textit{полной булевой функции}. Наиболее наглядно построение ДНФ с использованием каркасов, порожденных цепями булева куба.

Для дальнейших построений воспользуемся леммой Анселя о покрытии булева куба цепями, доказательство которой можно найти~в~\cite{ja10}:
\begin{lemma} Булев куб $B^n$ можно покрыть $\dbinom n{]n/2[}$ непересекающимися симметричными цепями $C_i$ так, что:
\begin{enumerate}
\item   число цепей длины $n - 2p+1$, где p = $1, . . . , \lceil n/2\rceil$, равно $\dbinom np~-~\dbinom n{p-1}$;
\item для любых трех вершин $v_1, v_2, v_3$, образующих цепь и принадлежащих одной и той же цепи $C_i$, дополнительная вершина принадлежит
цепи $C_j$ меньшей длины.
\end{enumerate}
\end{lemma}

А.~Г.~Дьяконов в работе \cite{dj02} рассматривал булевы функции заданные матрицами нулей имеющими единичную подматрицу размера $k\times k$, где $k$ число нулей функции. Для реализации функции $f(x_1,x_2,...,x_n)$, матрица нулей которой содержит единичную подматрицу в первых $k$ столбцах, А.~Г.~Дьяконов предложил использовать ДНФ
\begin{gather}\label{formula_identity_matrix}
D = \bigvee\limits_{1\le i<j\le k}\{x_ix_j\} \vee \bigvee\limits_{y=k+1}^n \left\{\{x_y\wedge \bigwedge_{s\in E(y)} \bar{x}_s\} \vee \{\bar{x}_y\wedge \bigwedge_{t\in Z(y)} x_t\}\right\}
\end{gather}

В работе \cite{dj02} было доказано, что указанная ДНФ действительно реализует функцию $f$ и число ее конъюнкций равно $2n + (k^2 - 5k)/2$. Заметим, что число литералов в указанной ДНФ равно $2n(k+1) - (k^2+3k)$. Указанная формула будет использована в доказательстве следующей теоремы.

\begin{theorema}\label{bound_upper}
Полная функция $F_k$ с числом нулей $k$ не менее 8 может быть реализована ДНФ ранга, не превосходящего~$3n + 6n/\log_2 n + 6n^{0.93}\log_2n$
\end{theorema}
\begin{proof}
Обозначим $M_{F_k}$ матрицу нулей указанной функции. Пусть тест $T_\lambda$ состоит из множества столбцов веса, не превосходящего $\lambda(k)$, $]k/4[\le\lambda(k) \le k/2$; точное значение $\lambda(k)$ будет выбрано позднее.

Множество цепей длины $m$ обозначим $\mathbb{C}(m)$. Для цепи $c_l = \{x_{i_1}^{\sigma_{i_1}}, x_{i_2}^{\sigma_{i_2}},...,x_{i_m}^{\sigma_{i_m}}\} \in \mathbb{C}(m)$ каркас $C_T(c(l))$ определим равенством
\[C_T(l) =
    \{y_1(c_l), y_2(c_l), x_{i_1}^{\sigma_{i_1}}\}
    \cup \bigcup\limits_{j=1}^{m-1}     \{x_{i_j}^{\sigma_{i_j}}, z_{i_j}, x_{i_{j+1}}^{\sigma_{i_{j+1}}\oplus1}\}
    \cup   \{x_{i_m}^{\sigma_{i_m}\oplus1}, y_3(c_l), y_4(c_l)\},
\]
где $y_1(c_l), y_2(c_l)$ литералы, образующие ортогональное разбиение $x_{i_1}^{\sigma_{i_1}\oplus 1}$; $y_3(c_l), y_3(c_l)$ тестовые литералы, образующие ортогональное разбиение $x_{i_m}^{\sigma_{i_m}}$. Литерал $z_{i_j}$ определяется равенством $z_{i_j} = x_{i_j}^{\sigma_{i_j}} \oplus x_{i_{j+1}}^{\sigma_{i_{j+1}}\oplus1}$.

Так как тест $T_\lambda$ содержит все литералы веса, не превосходящего $k/4$, литералы $y_1$ -- $y_4$ и $z_{i_1}$~--~$z_{i_{m-1}}$  могут быть выбраны тестовыми для любой цепи.

Каждому ребру $e$ каркаса может быть сопоставлена конъюнкция $K[e]$, состоящая из произведений отрицаний литералов, содержащихся в ребре. Таким образом, число литералов ДНФ $D_{C_T}$, построенной по каркасу $C_{T(c_l)}$, равна $3(m+1)$; каркас при этом содержит $2m$ внешних литералов.

По лемме \ref{hypergraph_span_union} ДНФ $D_\lambda$ реализует функцию $f$:
\[D = D_T \vee \bigvee\limits_{m\le k-2\lambda(k)}\bigvee\limits_{c_p \in\mathbb{C}_m}\bigvee\limits_{e\in C_T(c(l))}K[e],\]
где ДНФ $D_T$ построена по тесту $T$ в соответствии с формулой \ref{formula_identity_matrix}.

Число литералов $l(D_\lambda)$ равно
\begin{multline}
l(D_\lambda) =
    2(k+1)\left(\sum\limits_{i\le\lambda(k)} \dbinom ki\right) - (k^2 +3k) +
    3\left(\sum\limits_{i > \lambda(k)}\dbinom ki \right) +
    3\left(\dbinom k{]k/2[} - \dbinom k{\lambda}\right)
\end{multline}

Оценим полученное число сверху, используя энтропийное неравенство на биномиальные коэффициенты:
\[l(D_\lambda) < 3n + 3\cdot2^k/\sqrt{k} + 2(k+1)2^{kH(\lambda/k)}\]
Функция энтропии $H(x) = -x\log_2{x} - (1-x)\log_2{(1-x)}$ монотонно возрастает в $(0,1)$, следовательно выберем $\lambda$ наименьшим возможным, а именно, $\lambda = ]k/4[$, при $k > 8$ справедливо
\[l(D_\lambda) < 3n + 6n/\log_2 n + 6n^{0.93}\log_2n.\]
Полученная оценка завершает доказательство теоремы.
\end{proof}
\begin{Remark}
Приведенная конструкция в теореме не является оптимальной по числу используемых конъюнкций, однако обеспечивает нужную асимптотику и проста для пояснений.
\end{Remark}

Следствием теорем \ref{bound_lower} и \ref{bound_upper} является утверждение, устанавливающее асимптотику сложности полной булевой функции.
\begin{theorema}\label{bound_equal}
Полная функция может быть реализована ДНФ, содержащей $3n(1+o(1))$ литералов и $n(1+o(1))$ при $n\rightarrow\infty$; при этом ни одна из границ асимптотически не улучшаема.
\end{theorema}
\begin{example}
Приведем ДНФ получаемую в результате работы алгоритма для функции $G_6$, если в качестве тестовой выбрана единичная подматрица. Выберем ребра каркаса  следующим образом:
\begin{multline*}
  E(C_T) =  \biggl\{ \{  {x}_{1},{x}_{4},\bar{x}_{5} \};
    \{ {x}_{5},{x}_{8},\bar{x}_{13} \};
    \{ {x}_{13},{x}_{16},\bar{x}_{29} \};
    \{ {x}_{29},{x}_{2},{x}_{32} \};
    \{ {x}_{2},{x}_{1},\bar{x}_{3} \};
    \{ {x}_{3},{x}_{4},\bar{x}_{7} \};\\
    \{ {x}_{8},{x}_{7},\bar{x}_{15} \};
    \{ {x}_{16},{x}_{15},\bar{x}_{31} \};
    \{ {x}_{2},{x}_{4},\bar{x}_{6} \};
    \{ {x}_{6},{x}_{8},\bar{x}_{14} \};
    \{ {x}_{14},{x}_{16},\bar{x}_{30} \};
    \{ {x}_{30},{x}_{1},{x}_{32} \};
    \{ {x}_{2},{x}_{8},\bar{x}_{10} \};\\
    \{ {x}_{10},{x}_{1},\bar{x}_{11} \};
    \{ {x}_{11},{x}_{16},\bar{x}_{27} \};
    \{ {x}_{27},{x}_{4},\bar{x}_{32} \};
    \{ {x}_{8},{x}_{1},\bar{x}_{9} \};
    \{ {x}_{9},{x}_{16},\bar{x}_{25} \};
    \{ {x}_{25},{x}_{2},\bar{x}_{27} \};
    \{ {x}_{11},{x}_{1},\bar{x}_{12} \};\\
    \{ {x}_{12},{x}_{16},\bar{x}_{28} \};
    \{ {x}_{28},{x}_{2},\bar{x}_{30} \};
    \{ {x}_{16},{x}_{2},\bar{x}_{18} \};
    \{ {x}_{18},{x}_{1},\bar{x}_{19} \};
    \{ {x}_{19},{x}_{4},\bar{x}_{23} \};
    \{ {x}_{23},{x}_{1},\bar{x}_{24} \};
    \{ {x}_{16},{x}_{1},\bar{x}_{17} \};\\
    \{ {x}_{17},{x}_{4},\bar{x}_{21} \};
    \{ {x}_{21},{x}_{2},\bar{x}_{23} \};
    \{ {x}_{4},{x}_{16},\bar{x}_{20} \};
    \{ {x}_{20},{x}_{2},\bar{x}_{22} \};
    \{ {x}_{22},{x}_{1},\bar{x}_{23} \};
    \{ {x}_{24},{x}_{2},\bar{x}_{26} \};
    \{ {x}_{26},{x}_{4},\bar{x}_{30} \}
\biggr\}.
\end{multline*}
Для ДНФ $D$ реализующей функцию $F_6$ получаем
\begin{multline*}
D = x_1x_2 \vee x_1 x_4 \vee x_1x_8 \vee x_1x_{16} \vee x_1\bar{x}_{31} \vee
    x_2x_4 \vee x_2x_8 \vee x_2x_{16} \vee x_2\bar{x}_{31}  \vee
    x_4x_8 \vee x_4x_{16} \vee x_4\bar{x}_{31} \vee\\
    x_8x_{16} \vee x_{8}\bar{x}_{31} \vee
    x_{16}\bar{x}_{31} \vee
    \bar{x}_1    \bar{x}_2    \bar{x}_4    \bar{x}_8    \bar{x}_{16}    {x}_{31} \vee
    \bar{x}_{1}\bar{x}_{4}{x}_{5}  \vee
     \bar{x}_{5}\bar{x}_{8}{x}_{13} \vee
     \bar{x}_{13}\bar{x}_{16}{x}_{29} \vee
     \bar{x}_{29}\bar{x}_{2}{x}_{31} \vee
     \bar{x}_{2}\bar{x}_{1}{x}_{3} \vee\\
     \bar{x}_{3}\bar{x}_{4}{x}_{7} \vee
     \bar{x}_{8}\bar{x}_{7}{x}_{15} \vee
     \bar{x}_{16}\bar{x}_{15}{x}_{31} \vee
     \bar{x}_{2}\bar{x}_{4}{x}_{6} \vee
     \bar{x}_{6}\bar{x}_{8}{x}_{14} \vee
     \bar{x}_{14}\bar{x}_{16}{x}_{30} \vee
     \bar{x}_{30}\bar{x}_{1}{x}_{31} \vee
     \bar{x}_{2}\bar{x}_{8} {x}_{10} \vee
     \bar{x}_{10}\bar{x}_{1} {x}_{11} \vee\\
     \bar{x}_{11}\bar{x}_{16} {x}_{27} \vee
     \bar{x}_{27}\bar{x}_{4}{x}_{31} \vee
     \bar{x}_{8}\bar{x}_{1}{x}_{9} \vee
     \bar{x}_{9}\bar{x}_{16}{x}_{25} \vee
     \bar{x}_{25}\bar{x}_{2}{x}_{27} \vee
     \bar{x}_{11}\bar{x}_{1}{x}_{12} \vee
     \bar{x}_{12}\bar{x}_{16}{x}_{28} \vee
     \bar{x}_{28}\bar{x}_{2}{x}_{30} \vee\\
     \bar{x}_{16}\bar{x}_{2}{x}_{18} \vee
     \bar{x}_{18}\bar{x}_{1}{x}_{19} \vee
     \bar{x}_{19}\bar{x}_{4}{x}_{23} \vee
     \bar{x}_{23}\bar{x}_{1}{x}_{24} \vee
     \bar{x}_{16}\bar{x}_{1}{x}_{17} \vee
     \bar{x}_{17}\bar{x}_{4}{x}_{21} \vee
     \bar{x}_{21}\bar{x}_{2}{x}_{23} \vee
     \bar{x}_{4}\bar{x}_{16}{x}_{20} \vee\\
     \bar{x}_{20}\bar{x}_{2}{x}_{22} \vee
     \bar{x}_{22}\bar{x}_{1}{x}_{23} \vee
     \bar{x}_{24}\bar{x}_{2}{x}_{26} \vee
     \bar{x}_{26}\bar{x}_{4}{x}_{30}.
\end{multline*}
\end{example}

\section{Заключение}

В настоящей работе установлена асимптотически точная минимальная граница сложности ДНФ полной функции по числу литералов и конъюнкций. К построению ДНФ данной функции сводится построение ДНФ почти всех булевых функций с числом нулей, не превосходящим $\log_2 n - \log_2 \log_2 n + 1$ \cite{zk85dan,zk86}.

Предложена методика получения нижних на число литералов, содержащихся в ДНФ, которая может быть применена как к анализу конкретных ДНФ конкретных функций, так и к анализу сложности произвольных ДНФ широких классов функций.

Получен алгоритм построения ДНФ булевых функций специального вида, позволяющий строить дизъюнктивные нормальные формы достаточно малого ранга.

Показано, что применение указанного алгоритма для полной функции позволяет построить ДНФ, имеющую как асимптотически минимальное число конъюнкций, равное $n(1+o(1))$, при~$n\rightarrow\infty$, так и асимптотически минимальный ранг, равный $3n(1+o(1))$, при $n\rightarrow\infty$.


\bibliographystyle{apalike}
\bibliography{rank_jvm}

\begin{thebibliography}{}

\bibitem[Вебер, 1979]{veb79}
Вебер, К. (1979).
\newblock О различных понятиях минимальности дизъюнктивных нормальных форм.
\newblock {\em Проблемы кибернетики}, 36:129--158.

\bibitem[Яблонский, 2010]{ja10}
Яблонский, С.~В. (2010).
\newblock {\em Введение в дискретную математику: Учебное пособие для вузов,
  6-ое изд.}
\newblock М: Высшая школа.

\bibitem[Дьяконов, 2001]{dj01}
Дьяконов, А.~Г. (2001).
\newblock Реализация одного класса булевых функций с малым числом нулей
  тупиковыми дизъюнктивными нормальными формами.
\newblock {\em Журнал вычислительной математики и математической физики},
  41(5):828--835.

\bibitem[Дьяконов, 2002]{dj02}
Дьяконов, А.~Г. (2002).
\newblock Тестовый подход к реализации дизъюнктивными нормальными формами
  булевых функций с малым числом нулей.
\newblock {\em Журнал вычислительной математики и математической физизики},
  42(6):924--928.

\bibitem[Boros et~al., 2011]{boros11}
Boros, E., Crama, Y., Hammer, P.~L., Ibaraki, T., Kogan, A., and Makino, K.
  (2011).
\newblock Logical analysis of data: classification with justification.
\newblock {\em Annals of Operations Research}, 188(1):33--61.

\bibitem[Boros et~al., 2000]{boros00}
Boros, E., Hammer, P.~L., Ibaraki, T., Kogan, A., Mayoraz, E., and Muchnik, I.
  (2000).
\newblock An implementation of logical analysis of data.
\newblock {\em IEEE Transactions on Knowledge and Data Engineering},
  12:292--306.

\bibitem[Коршунов, 2009]{kor09}
Коршунов, А.~Д. (2009).
\newblock Некоторые нерешенные задачи дискретной математики и математической
  кибернетики.
\newblock {\em Успехи математических наук}, 64(5):3--20.

\bibitem[Коршунов, 2012]{kor12}
Коршунов, А.~Д. (2012).
\newblock Сложность вычислений булевых функций.
\newblock {\em Успехи математических наук}, 68(3):97--168.

\bibitem[Кудрявцев and Андреев, 2009]{ka09}
Кудрявцев, В.~Б. and Андреев, А.~Е. (2009).
\newblock О сложности алгоритмов.
\newblock {\em Фундаментальная и прикладная математика}, 15(3):135--181.

\bibitem[Golea et~al., 1997]{Golea97}
Golea, M., Bartlett, P.~L., Mason, L., and et~al. (1997).
\newblock Generalization in decision trees and dnf: Does size matter?
\newblock In {\em Advances in Neural Information Processing Systems}, pages
  259--265. The MIT Press.

\bibitem[Максимов, 2012a]{me12a}
Максимов, Ю.~В. (2012a).
\newblock Простые дизъюнктивные нормальные формы булевых функций с ограниченным
  числом нулей.
\newblock {\em Доклады академии наук}, 445(2):143--145.

\bibitem[Максимов, 2012b]{me12b}
Максимов, Ю.~В. (2012b).
\newblock Сравнительный анализ сложности булевых функций с малым числом нулей.
\newblock {\em Доклады Академии Наук}, 447(6):607--609.

\bibitem[Максимов, 2013]{me13}
Максимов, Ю.~В. (2013).
\newblock Реализация булевых функций с ограниченным числом нулей в классе
  дизъюнктивных нормальных форм.
\newblock {\em Журнал вычислительной математики и математической физики},
  53(9):1569--1588.

\bibitem[Mubayi et~al., 2006]{mu06}
Mubayi, D., Turan, G., and Zhao, Y. (2006).
\newblock The dnf exception problem.
\newblock {\em Theoretical Computer Science}, 352(1--3):85--96.

\bibitem[Журавлёв and Коган, 1985]{zk85dan}
Журавлёв, Ю. and Коган, А. (1985).
\newblock Реализация булевых функций с малым числом нулей дизъюнктивными
  нормальными формами и смежные задачи.
\newblock {\em Доклады АН СССР}, 285(4):795--799.

\bibitem[Журавлёв, 1958]{zh58}
Журавлёв, Ю.~И. (1958).
\newblock Об отделимости подмножеств вершин $n$-мерного единичного куба.
\newblock {\em Труды МИАН СССР}, 51:143–157.

\bibitem[Журавлёв, 1960]{zmin}
Журавлёв, Ю.~И. (1960).
\newblock О различных понятиях минимальности дизъюнктивных нормальных форм.
\newblock {\em Сибирский математический журнал}, 1(4):609–610.

\bibitem[Журавлёв, 1978]{zh78}
Журавлёв, Ю.~И. (1978).
\newblock Об алгебраическом подходе к решению задач распознавания или
  классификации.
\newblock {\em Проблемы кибернетики}, 33:5--68.

\bibitem[Журавлёв and Коган, 1986]{zk86}
Журавлёв, Ю.~И. and Коган, А.~Ю. (1986).
\newblock Алгоритм построения дизъюнктивной нормальной формы, эквивалентной
  произведению левых частей булевых уравнений нельсоновского типа.
\newblock {\em Журнал вычислительной математики и математической физики},
  26(8):1243--1249.

\end{thebibliography}

\end{document}